\documentclass[10.5pt,a4paper]{amsart}
\usepackage{amsfonts}
\usepackage{amssymb}
\usepackage{amsthm}
\usepackage[english]{babel}
\usepackage{hhline}
\usepackage[ansinew]{inputenc}
\usepackage{amsmath}
\usepackage[all,2cell,ps]{xy}
\usepackage{qsymbols}
\usepackage{color}
\usepackage{epsfig}
\usepackage{color}
\usepackage{graphics}
\usepackage{graphicx}%
\usepackage{enumerate}

\usepackage{amssymb}

\newcommand{\ra}{\rightarrow}
\newcommand{\rla}{\leftrightarrow}
\newcommand{\we}{\wedge}
\newcommand{\cd} {\mathrm{\cdot}}
\newcommand{\edge}[1]{\ar@{-}[#1]}

\newcommand{\V}{\mathsf{DLCMI}}
\newcommand{\Var}{\mathsf{V}}

\newcommand{\GCRL}{\mathsf{GCRL}}
\newcommand{\CWRL}{\mathsf{CWRL}}
\newcommand{\WH}{\mathsf{WH}}
\newcommand{\DL}{\mathsf{DL}}
\newcommand{\ConA}{\mathrm{Con}(A)}
\newcommand{\IDCRL} {\mathsf{IDCRL}}
\newcommand{\MV}{\mathsf{MV}}

\theoremstyle{plain}
\newtheorem{thm}{Theorem}
\newtheorem{cor}[thm]{Corollary}
\newtheorem{lem}[thm]{Lemma}
\newtheorem{prop}{Proposition}

\theoremstyle{definition}
\newtheorem{defn}{Definition}
\newtheorem{rem}{Remark}

\newtheorem{ex}{Example}
\begin{document}

\title{On principal congruences in distributive lattices
with a commutative monoidal operation and an implication
\\
}

\author{Ramon Jansana and Hern\'an Javier San Mart\'in}

\maketitle

\begin{abstract}
In this paper we introduce and study a variety of algebras that
properly includes integral distributive commutative residuated
lattices and weak Heyting algebras. Our main goal is to give a
characterization of the principal congruences in this variety. We
apply this description in order to study compatible functions.
\end{abstract}

\section{Introduction}

It is very convenient to have good descriptions of the principal
congruences of the algebras of a variety. One type of description
is having first-order definable principal congruences. A much
simpler and useful type of description is having equationally
definable principal congruences. This concept was introduced in
\cite{FrGrQu79,FrGrQu80}. Recall that a variety $\Var$ has
\textit{equationally definable principal congruences} (EDPC) if
there exists a finite family of quaternary terms
$\{u_i,v_i\}_{i=1}^{r}$ such that for every algebra $A$ in $\Var$
and every principal congruence $\theta(a,b)$ of $A$\footnote{If
$A$ is an algebra and $a,b\in A$, then $\theta(a,b)$ denotes the
principal congruence of $A$ generated by $(a, b)$, i.e., the
smallest congruence of $A$ that contains $(a,b)$.}, it holds that
$(c,d) \in \theta(a,b)$ if and only if $u_i(a, b, c, d) = v_i(a,
b, c, d)$ for each $i = 1,\dots, r$. The property EDPC is also of
logical interest because an algebraizable logic whose equivalent
algebraic semantics is a variety $\Var$ has some form of
deduction-detachment theorem if and only if the variety $\Var$ has
EDPC, a consequence of a more general result proved by Blok and
Pigozzi in \cite[Thm.\ 5.5]{BP}(see also \cite[Thm.\ 3.85]{Fo16}).

There are varieties that do not have EDPC but where it is still
possible to have a good characterization of the principal
congruences with the following local version of EDPC: there exists
a finite family of quaternary terms
$\{u_{(i,n,k)},v_{(i,n,k)}\}_{i=1}^{r}$ (with $n, k\geq 0$) such
that for every principal congruence $\theta(a,b)$ of any algebra
$A$ in the variety it holds that $(c,d) \in \theta(a,b)$ if and
only if there exist $n,k\geq 0$ such that $u_{(i,n,k)}(a, b, c, d)
= v_{(i,n,k)}(a, b, c, d)$ for every $i = 1, \dots r$. We say that
a variety of algebras has \emph{locally equationally definable
principal congruences} if there exists a finite family of
quaternary terms such that in every algebra of the variety they
define the principal congruences in the way just described. In
particular, the variety of commutative residuated lattices has
locally equationally definable principal congruences (it can be
deduced from results of \cite{Ag}) and the variety of weak Heyting
algebras has this property too \cite[Theorem 2.2]{SMc}. The concept of 
locally equationally definable principal congruences for quasivarieties is introduced in 
\cite{BP91}, in an equivalent form. From the logical point of view it is related to the notion of having a local deduction theorem. 

The main goal of the paper is to prove that a variety of algebras
that properly includes the integral distributive commutative
residuated lattices \cite{Ts} and the weak Heyting algebras
\cite{CJ, Ben1} has locally equationally definable principal
congruences. We call the members of this variety
\emph{distributive lattices with a commutative monoidal operation
and an implication}  and we denote the variety by $\V$.

The second goal of the paper is the study of compatible functions
of algebras in $\V$, by applying the characterization  of the
principal congruences. Given an algebra $A$ and a function
$f:A^{n}\ra A$, we say that $f$ is \emph{compatible} if every
congruence of $A$ is a congruence of $A$ enriched with $f$ as a
new operation. The principal congruences are closely related to
compatible functions. For instance, if $f:A\ra A$ is a function,
then $f$ is compatible if and only if $(f(a),f(b))\in \theta(a,b)$
for every $a,b \in A$. Furthermore, certain algebraizable logics
whose equivalent algebraic semantics are varieties are also
connected with compatible functions. Caicedo showed in \cite{C1}
that in any axiomatic expansion $\mathrm{L}'$ of an algebraizable
logic $\mathrm{L}$ by adding only new connectives that are
implicitly definable, the new connectives can be translated to
compatible functions (whenever they exist) in the algebras of the
equivalent algebraic semantics of the initial logic $\mathrm{L}$.
Therefore, the principal congruences in a variety that is the
equivalent algebraic semantics of an algebraizable logic
$\mathrm{L}$ (which has an axiomatic expansion $\mathrm{L}'$ with
the above mentioned property) is also strongly linked with
properties of $\mathrm{L}'$.

The paper is organized as follows. In Section 2 we provide the
basic definitions and results. In Section 3 we show that $\V$ has
locally equationally definable principal congruences. In
particular, we obtain known characterizations of the principal
congruences of weak Heyting algebras and integral distributive
commutative residuated lattices. The first one was obtained in
\cite{SMc}, and the second is part of the folklore and it follows
easily from \cite{Ag}. In Section 4 we apply the results of the
previous section to study compatible functions in $\V$. Finally,
in Section 5 we establish other connections with existing
literature.

\section{Basic definitions and results}

In what follows we start by recalling the definitions of
commutative residuated lattice and weak Heyting algebra
respectively.

\begin{defn}
An algebra $(A,\we,\vee,\cd,\ra,e)$ of type $(2,2,2,2,0)$ is said
to be a \emph{commutative residuated lattice} if the following
conditions are satisfied:
\begin{enumerate}[\normalfont \hspace{4mm} (1)]
\item $(A,\cd,e)$ is a commutative monoid,
\item $(A,\we,\vee)$ is a lattice,
\item   for every $a,b,c \in A$, $a\cd b
\leq c$ if and only if $a\leq b\ra c$.
\end{enumerate}
\end{defn}

A commutative residuated lattice $(A,\we,\vee,\cd,\ra,e)$ is
\textit{distributive} if its lattice reduct is distributive and it
is \textit{integral} if the unit of the monoid is the largest element
of the lattice reduct. We write $\IDCRL$ for the variety of
integral distributive commutative residuated lattices. Since the
class of commutative residuated lattices is a variety, then
$\IDCRL$ is a variety.

\begin{defn}
An algebra $(A,\we,\vee,\ra,0,1)$ of type $(2,2,2,0,0)$ is a
\textit{weak Heyting algebra} if the reduct algebra $(A,\wedge,
\vee, 0,1)$ is a bounded distributive lattice and $\ra$ is a
binary operation such that satisfies the following conditions for
all $a,b,c\in A$:
\begin{enumerate}[\normalfont \hspace{4mm} (1)]

\item $(a\ra
b)\we(a\ra c)=a\ra(b\we c)$,
\item $(a\ra c)\we(b\ra c)=(a\vee b)\ra c$,
\item $(a\ra b)\we(b\ra c)\leq a\ra c$,
\item $a\ra a = 1$.
\end{enumerate}
\end{defn}

We denote by $\WH$ the variety of weak Heyting algebras.
\vspace{1pt}

In what follows we introduce a variety that properly contains the
variety of integral distributive commutative residuated lattices
and the variety of weak Heyting algebras. This is the variety that
we study in this paper.

\begin{defn} \label{d1}
An algebra $(A,\we,\vee,\cd,\ra,1)$ of type $(2,2,2,2,0)$ is a
\textit{distributive lattice with a commutative monoidal operation
and an implication} if for every $a,b,c\in A$ the following
conditions are satisfied:
\begin{enumerate}[\normalfont \hspace{4mm} (1)]
\item $(A,\we,\vee)$ is a distributive lattice, \item $1$ is the
largest element of $(A,\we,\vee)$, \item $(A,\cd,1)$ is a commutative
monoid, \item $(a\ra b)\we(a\ra c)=a\ra(b\we c)$, \item $(a\ra
c)\we(b\ra c)=(a\vee b)\ra c$, \item $a\ra a = 1$, \item $(a\vee
b)\cd c = (a\cd c) \vee (b \cd c)$, \item $(a\ra b)\cd(b\ra c)\leq
a\ra c$, \item $a\ra b \leq (a\cd c) \ra (b\cd c)$.
\end{enumerate}
\end{defn}

We denote by $\V$ the variety of algebras given in Definition
\ref{d1} and  we refer to its members as $\V s$.

Let $A\in \V$. The following conditions are satisfied for all
$a,b,c\in A$:

\begin{enumerate}[\normalfont \hspace{4mm} 1)]
\item If $a\leq b$, then $a\cd c\leq b\cd c$, $b\ra c \leq a\ra c$,
and $c\ra a \leq c\ra b$. \item $a\cd b \leq a\we b$. \item $a\cd
b \leq a$. \item $1\ra a \leq b \ra (a\cd b)$.
\end{enumerate}

\begin{lem}
$\IDCRL$ is a subvariety of $\V$.
\end{lem}

\begin{proof}
It follows from straightforward computations based on properties
of integral commutative residuated lattices \cite{Ts}.
\end{proof}

\begin{lem} \label{WHsV}
The variety $\WH$ can be seen as a subvariety of $\V$, namely the
subvariety that is defined by the equation $x \wedge y \approx x
\cd y$. More precisely, we have that $(A,\we,\vee,\ra,0,1) \in
\WH$ if and only if $(A,\we,\vee,\we,\ra,1) \in \V$.
\end{lem}

\begin{proof}
In order to show that $\WH$ can be seen as a subvariety of $\V$,
we will prove that the inequality $a\ra b \leq (a\we c) \ra (b\we
c)$ holds in weak Heyting algebras, since the rest of the items
can be proved easily. Since $a\we c\leq a$, then $a\ra b \leq (a\we
c)\ra b$. But $(a\we c) \ra (b\we c) = (a\we c) \ra b$, so we
conclude that $a\ra b \leq (a\we c)\ra (b\we c)$.
\end{proof}

Now we will give an example of an algebra in $\V$ that is neither
a weak Heyting algebra nor an  integral distributive commutative
residuated lattice.

\begin{ex} \label{ex1}
Let $H_3 = \{0,a,1\}$ be the chain of three elements with $0<a<1$.
The algebra $(H_3,\we,\vee,\ra,1)$ belongs to $\WH$, where $\ra$
is given by $x\ra y = 1$ for every $x,y\in H_3$. Following the
convention given in Lemma \ref{WHsV} we also can say that
$H_{3}^{\WH} = (H_3, \we,\vee,\we,\to, 1) \in \WH$. Let
$H_{3}^{\MV} = (H_3,\we,\vee,\odot,\ra,1)$ be the MV-chain of
three elements presented as a residuated lattice \cite{CMO}. The
product and implication $\ra$ in $H_{3}^{\MV}$ are given in the
following tables:

\vspace{10pt}
\begin{center}
\begin{tabular}{c|c c c }
  $\odot$           &$0$    & $a$    &    $1$          \\
  \hline
  $0$           & $0$   & $0$    &    $0$         \\
  $a$           & $0$   & $0$    &    $a$           \\
  $1$           & $0$   & $a$    &    $1$
\end{tabular}
\hskip15pt
\begin{tabular}{c|c c c }
  $\ra$           &$0$    & $a$    &    $1$         \\
  \hline
  $0$           & $1$   & $1$    &    $1$          \\
  $a$           & $a$   & $1$    &    $1$           \\
  $1$           & $0$   & $a$    &    $1$
\end{tabular}
\end{center}
\vspace{3pt}

We define $A$ as the product algebra $H_{3}^{\MV} \times
H_{3}^{\WH}$. Since $H_{3}^{\MV}, H_{3}^{\WH} \in \V$ and $\V$ is
a variety,   then $A\in \V$. Since $(a,a)\cd (a,a) = (0,a)$,
$(a,a)\we (a,a) = (a,a)$ and $0\neq a$, then $(a,a)\cd (a,a) \neq
(a,a)\we (a,a)$, so $A\notin \WH$. In order to show that $A\notin
\IDCRL$ notice that $(x,y)\cd (z,w) = (x\odot z,y\we w)$ for every
$x,y,z,w \in H_3$. Since $(a,a) \ra (0,0) = (a,1)$ then $(a,a)
\leq (a,a)\ra (0,0)$. However $(a,a) \cd (a,a) = (0,a)\nleq
(0,0)$. Therefore, $A\notin \IDCRL$.
\end{ex}

It follows from Example \ref{ex1} that $\WH$ and $\IDCRL$ are
proper subvarieties of $\V$. Besides, $\WH$ and $\IDCRL$ are
incomparable varieties. In order to show it, first note that
$H_{3}^{\MV} \in \IDCRL$ and $H_{3}^{\MV} \notin \WH$ because
$a\odot a = 0$ and $a\we a \neq 0$. Finally note that $H_{3}^{\WH}
\in \WH$ and $H_{3}^{\WH} \notin \IDCRL$, which follows from the
fact that $1\leq 1 = 1 \ra a$ and $1\we 1 = 1 \nleq a$.
\vspace{3pt}

We have the following picture: 
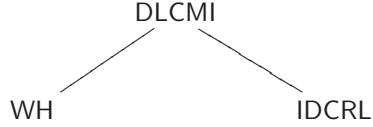
\vspace{3pt}

\centerline{
 \xymatrix{
            &  & \V \\
            & \WH \edge{ur} & & \IDCRL \edge{ul}
} }
\vspace{3pt}

The following elementary lemma allows us to give an alternative
presentation for $\V s$.

\begin{lem} \label{lemaprod}
Let $A$ be an algebra of type $(2,2,2,2,0)$ which satisfies the
conditions $1), \ldots,  8)$ of Definition \ref{d1}. The following
conditions are equivalent:
\begin{enumerate}[\normalfont \hspace{2mm} 1)]
\item $A$ satisfies condition $9)$ of Definition \ref{d1}.
\item For every $a,b,c,d\in A$, $(a \ra b)\cd (c\ra d)\leq (a\cd
c)\ra (b\cd d)$.
\end{enumerate}
\end{lem}

\begin{proof}
Assume  condition $9)$. We have that $a\ra b \leq (a\cd c) \ra
(b\cd c)$ and $c\ra d \leq (b \cd c) \ra (b\cd d)$. Then
\[
\begin{array}
[c]{lllll}
 (a\ra b)\cd (c\ra d) & \leq & ((a\cd c) \ra (b\cd c))\cd ((b\cd c) \ra (b\cd d)) &  & \\
 & \leq & (a\cd c)\ra (b\cd d). &  &
\end{array}
\]
Conversely, suppose that for every $a,b,c,d \in A$ the inequality
$(a \ra b)\cd (c\ra d)\leq (a\cd c)\ra (b\cd d)$ is satisfied.
Then
\[
\begin{array}
[c]{lllll}
 a\ra b & = & (a\ra b)\cd 1 &  & \\
  & = & (a\ra b)\cd (c\ra c)  &  &\\
 & \leq & (a\cd c)\ra (b\cd c), &  &
\end{array}
\]
which was our aim.
\end{proof}

\section{Principal congruences} \label{SPC}

In \cite{CJ} it was proved that $\WH$ does not have EDPC. Hence,
$\V$ does not have EDPC either. In this section we prove that
$\V$ has locally equationally definable principal congruences.

\begin{defn}
A variety of algebras $\mathsf{V}$ has \emph{locally equationally
definable principal congruences} if there exists a finite family
of quaternary terms $\{u_{(i,n,k)},v_{(i,n,k)}\}_{i=1}^{r}$ (with
$n, k\geq 0$) such that for every $A \in \mathsf{V}$ and $a, b \in
A$ it holds that for every $c, d \in A$, $(c,d) \in \theta(a,b)$
if and only if there exist $n,k\geq 0$ such that $u_{(i,n,k)}(a,
b, c, d) = v_{(i,n,k)}(a, b, c, d)$ for every $i = 1, \dots r$.
\end{defn}

We
start with some preliminary definitions and technical results.
\vspace{1pt}

If $A$ is an algebra, we denote by $\ConA$ the lattice of the
congruences of $A$. We refer by $\DL$  to the variety of distributive
lattices \cite{BD,BS}. The following lemma
is part of the folklore of distributive lattices.

\begin{lem} \label{l1}
Let $A \in \DL$, $\theta \in \ConA$ and $a,b \in A$. If there is
$c\in A$ such that $(a\we c,b\we c) \in \theta$ and $(a\vee c,
b\vee c) \in \theta$, then $(a,b) \in \theta$.
\end{lem}

\begin{lem} \label{l2}
Let $A\in \DL$, $\theta \in \ConA$, $(a,b) \in \theta$ and $c, d \in A$.
\[
\text{If
$(c\we a \we b,d\we a\we b), (c\vee a\vee b,d\vee a
\vee b)\in \theta$,  then $(c,d)\in \theta$.}
\]
\end{lem}

\begin{proof}
Since $(a,b) \in \theta$, then $(a\we b,b)\in \theta$ and $(a\vee
b,b) \in \theta$. Thus, $(a\we b,a\vee b) \in \theta$. This
implies that  $(c\vee (a\we b), c\vee a\vee b) \in
\theta$ and $(d\vee (a\we b), d\vee a \vee b)\in \theta$. Since by
hypothesis $(c\vee a\vee b,d\vee a \vee b) \in \theta$ then
\begin{equation} \label{dl1}
(c\vee (a\we b),d\vee (a\we b))\in \theta.
\end{equation}
By hypothesis we also have that
\begin{equation} \label{dl2}
(c\we (a\we b),d\we (a\we b)) \in \theta.
\end{equation}
Therefore, it follows from (\ref{dl1}), (\ref{dl2}) and Lemma
\ref{l1} that $(c,d)\in \theta$.
\end{proof}

Let $A\in \V$, $n\geq 1$ and $a, b\in A$. We define inductively
$a^n$ by setting $a^{0} := 1$ and $a^{n} := a \cd a^{n-1}$. We
also define $\square^{0}(a) = a$, $\square(a) = 1\ra a$ and the
iterated operation $\square^{n}$ in the usual way. We also define
\[
a\rla b := (a\ra b) \we (b\ra a),
\]
\[
t_{n}(a,b) := \square^{0}(a\rla b) \we \square(a\rla b)\we \cdots
\we \square^{n}(a\rla b).
\]

\begin{rem}
\label{rem:remark-1}
Let $A\in \V$, $n\geq 1$ and $a,b\in A$.
\begin{enumerate} [\normalfont \hspace{2mm} a)]
\item If $k$ is a natural number, we write $t_{n}^{k}(a,b)$ in
place of $(t_{n}(a,b))^{k}$. \item The map $\square$ preserves
finite meets. In particular $\square$ is monotonic, i.e., if
$a\leq b$ then $\square(a)\leq \square(b)$.
\end{enumerate}
\end{rem}

We will prove in Theorem \ref{PT} that  $\V$ has locally
equationally definable principal congruences by the set QT of the
following quaternary terms: \vspace{2mm}

\begin{itemize}
\item[-] $u_{1,n,k}(x_1,x_2,y_1,y_2) := (t_{n}^{k}(x_1,x_2) \cd
(y_1\we x_1 \we x_2)) \vee (y_2 \we x_1 \we x_2)$, \item[-]
$v_{1,n,k}(x_1,x_2,y_1,y_2) := y_2 \we x_1 \we x_2$, \item[-]
$u_{2,n,k}(x_1,x_2,y_1,y_2) := (t_{n}^{k}(x_1,x_2) \cd (y_2\we x_1
\we x_2)) \vee (y_1 \we x_1 \we x_2)$, \item[-]
$v_{2,n,k}(x_1,x_2,y_1,y_2) := y_1 \we x_1 \we x_2$, \item[-]
$u_{3,n,k}(x_1,x_2,y_1,y_2) := (t_{n}^{k}(x_1,x_2) \cd (y_1\vee
x_1 \vee x_2)) \vee (y_2 \vee x_1 \vee x_2)$, \item[-]
$v_{3,n,k}(x_1,x_2,y_1,y_2) := y_2 \vee x_1 \vee x_2$, \item[-]
$u_{4,n,k}(x_1,x_2,y_1,y_2) := (t_{n}^{k}(x_1,x_2) \cd (y_2\vee
x_1\vee x_2)) \vee (y_1 \vee x_1 \vee x_2)$, \item[-]
$v_{4,n,k}(x_1,x_2,y_1,y_2) := y_1 \vee x_1 \vee x_2$, \item[-]
$u_{5,n,k}(x_1,x_2,y_1,y_2) := t_{n}^{k}(x_1,x_2) \vee (y_1\rla
y_2)$, \item[-] $v_{5,n,k}(x_1,x_2,y_1,y_2) := y_1\rla
y_2$.\vspace{1mm}
\end{itemize}
with $n, k$ natural numbers.
\vspace{3mm}

To achieve our goal,  for any $A\in \V$ and $a,b\in A$ we define
the binary relation $R(a,b)$ as follows: $(c,d) \in R(a,b)$ if and
only if there are natural numbers $n$ and $k$ that satisfy the
following conditions:
\begin{enumerate}
\item[(C1)] $t_{n}^{k}(a,b)\cd (c \we a \we b) \leq d \we a \we b$ and
$t_{n}^{k}(a,b)\cd (d \we a \we b) \leq c \we a \we b$,
\item[(C2)]
$t_{n}^{k}(a,b)\cd (c \vee a \vee b) \leq d \vee a \vee b$ and
$t_{n}^{k}(a,b)\cd (d \vee a \vee b) \leq c \vee a \vee b$,
\item[(C3)]
$t_{n}^{k}(a,b) \leq c\rla d$.
\end{enumerate}
We say that $(n,k)$ is \textit{a pair of natural numbers associated with}
$(c,d)$.

\begin{rem} \label{r1}
Let $A\in \V$ and $a,b\in A$.
\begin{enumerate}[\normalfont \hspace{2mm}  a)]
\item $t_{n+1}(a,b) \leq t_n(a,b)$ and $t_{n}^{k+1}(a,b) \leq
t_{n}^{k}(a,b)$. Hence, $t_{m}^{p}(a,b) \leq t_{n}^{q}(a,b)$
whenever $n\leq m$ and $q\leq p$. \item Let $(c,d), (c', d')  \in
R(a,b)$ and $(p,q)$, $(r,s)$ pairs of natural numbers associated
with $(c,d)$ and $(c',d')$ respectively. Consider $n:=
\mathrm{max}\{p,r\}$ and $k :=\mathrm{max}\{q,s\}$. It  follows
from the previous item that $(n,k)$ is a pair of natural numbers
associated with $(c,d)$ and $(c',d')$.
\end{enumerate}
\end{rem}

Let $A\in \DL$ and $\theta$ an equivalence relation on $A$. It is
easily seen that in this case  $\theta$ is a
congruence if and only if  the following two conditions hold:
\begin{enumerate}[\normalfont \hspace{2mm} 1)]
\item For every $a,b,c\in A$, if $(a,b)\in \theta$, then $(a\we
c,b\we c)\in \theta$. \item For every $a,b,c\in A$, if $(a,b)\in
\theta$, then $(a\vee c,b\vee c)\in \theta$.
\end{enumerate}

In the next lemma we will use the above mentioned property.

\begin{lem} \label{l3}
Let $A\in \V$ and $a,b \in A$. Then $(a,b) \in R(a,b)$ and
$R(a,b)$ is a congruence of the lattice reduct of $A$.
\end{lem}

\begin{proof}
First we prove that $(a,b) \in R(a,b)$. Note that since  $a\rla b
\leq 1$, then $(a\rla b) \cd (a\we b) \leq a\we b$ and $(a\rla b)
\cd (a\vee b) \leq a\vee b$. Besides $a\rla b \leq a\rla b$. Thus
we have that conditions (C1), (C2) and (C3) in the definition of
$R(a, b)$ hold for $(a, b)$. Therefore, $(a,b) \in R(a,b)$. Now we
prove that $R(a,b)$ is a congruence of the lattice reduct of $A$.

(i) We will prove that $R(a,b)$ is an equivalence relation. The
reflexivity and the symmetry are immediate. In order to prove the
transitivity, consider $(c,d) \in R(a,b)$ and $(d,e) \in R(a,b)$.
It follows from Remark \ref{r1} that there is a pair of natural
numbers $(n,k)$ associated with $(c,d)$ and $(d,e)$. Then
\[
\begin{array}
[c]{lllll}
 t_{n}^{2k}(a,b)\cd (c\we a \we b)& = & t_{n}^{k}(a,b)\cd t_{n}^{k}(a,b)\cd (c\we a \we b) &  & \\
 & \leq & t_{n}^{k}(a,b)\cd (d\we a \we b)  &  &\\
 & \leq & e\we a\we b. &  &
\end{array}
\]
In a similar way it can be proved that $t_{n}^{2k}(a,b)\cd (e\we a
\we b)\leq  c\we a\we b$, $t_{n}^{2k}(a,b)\cd (c\vee a \vee b)
\leq e\vee a \vee b$,  and $t_{n}^{2k}(a,b)\cd (e\vee a \vee b) \leq
c\vee a \vee b$. Finally we will see that $t_{n}^{2k}(a,b) \leq
c\rla e$. Note that $t_{n}^{k}(a,b) \leq c\ra d$,
$t_{n}^{k}(a,b)\leq d\ra c$, $t_{n}^{k}(a,b)\leq d\ra e$ and
$t_{n}^{k}(a,b)\leq e\ra d$. Then
\[
\begin{array}
[c]{lllll}
 t_{n}^{2k}(a,b) & \leq &  (c\ra d)\cd(d\ra e) &  & \\
 & \leq & c\ra e. &  &
\end{array}
\]
In a similar way we can show that $t_{n}^{2k}(a,b) \leq e\ra c$.
Thus,
\[
t_{n}^{2k}(a,b) \leq (c\ra e)\we (e\ra c).
\]
Hence, $R(a,b)$ is a transitive relation.

(ii) Let $(c,d) \in R(a,b)$. We will prove that $(c\we e,d\we e)
\in R(a,b)$. We have that
\[
\begin{array}
[c]{lllll}
 t_{n}^{k}(a,b)\cd ((c\we e)\we a \we b) & \leq & t_{n}^{k}(a,b)\cd (c\we a \we b) &  & \\
 &\leq & d \we a \we b &  &
\end{array}
\]
and $ t_{n}^{k}(a,b)\cd ((c\we e)\we a \we b) \leq e$. Thus,
$t_{n}^{k}(a,b)\cd ((c\we e)\we a \we b) \leq (d\we e)\we a\we b$.
The same argument shows that $t_{n}^{k}(a,b)\cd ((d\we e)\we a \we
b) \leq (c\we e)\we a\we b$.

Now we will prove the inequality $t_{n}^{k}(a,b)\cd ((c\we e)\vee
a \vee b) \leq (d\we e)\vee a\vee b$. First note that
\[
\begin{array}
[c]{lllll}
 t_{n}^{k}(a,b)\cd ((c\we e)\vee a \vee b) & \leq & t_{n}^{k}(a,b)\cd (c\vee a \vee b) &  & \\
 &\leq & d \vee a \vee b. &  &
\end{array}
\]
We also have that
\[
t_{n}^{k}(a,b)\cd ((c\we e)\vee a \vee b)  \leq e \vee a \vee b.
\]
Then,
\[
t_{n}^{k}(a,b)\cd ((c\we e)\vee a \vee b) \leq (d\vee a \vee b)\we
(e\vee a \vee b).
\]
Taking into account the distributivity of the underlying lattice
of $A$ we obtain
\[
(d\vee a \vee b)\we (e\vee a \vee b) = (d\we e) \vee a \vee b.
\]
Hence, $t_{n}^{k}(a,b)\cd ((c\we e)\vee a \vee b) \leq (d\we e)
\vee a \vee b$. We also have that
\[
t_{n}^{k}(a,b)\cd ((d\we e)\vee a \vee b) \leq (c\we e) \vee a
\vee b.
\]
In what follows we will prove the inequality $t_{n}^{k}(a,b) \leq
(c\we e) \rla (d\we e)$. It is enough to prove the inequality
$c\rla d \leq (c\we e)\rla (d\we e)$. First note that $(c\we e)
\ra (d\we e) = ((c\we e)\ra d)\we ((c\we e)\ra e)$. Since $c\we e
\leq c$, then $(c\we e) \ra d \geq c\ra d$ and since $c\we e \leq
e$, then $(c\we e) \ra e \geq e\ra e = 1$. Hence,
\[
\begin{array}
[c]{lllll}
 (c\we e)\ra (d\we e) & = & ((c\we e)\ra d)\we ((d\we e)\ra e)  &  & \\
 & \geq & c\ra d. &  &
\end{array}
\]
Analogously we obtain that $(d\we e)\ra (c\we e) \geq d\ra c$.
Thus $c\rla d\leq (c\we e)\rla (d \we e)$. Thus, we have proved
that $(c\we e,d\we e) \in R(a,b)$.

(iii) The fact $(c\vee e,d \vee e)\in R(a,b)$ whenever $(c,d)\in
R(a,b)$ can be proved using similar ideas to the ones employed for
the case (ii). Now in addition we use the property $(c\vee d)\cd e
= (c \cd e) \vee (d\cd e)$ for every $c,d,e \in A$.
\end{proof}

The following lemma will play a fundamental role.

\begin{lem} \label{il}
Let $A\in \V$ and $a\in A$. Then $(1\ra a)^{n} \leq 1\ra a^{n}$
for every $n$.
\end{lem}

\begin{proof}
Notice that for every $n$ the we have that
\[
(1\ra a^{n})\cd (1\ra a) \leq (1\cd 1) \ra (a \cd a^{n}),
\]
i.e.,
\begin{equation} \label{feq}
(1\ra a^{n})\cd (1\ra a) \leq 1 \ra a^{n+1}.
\end{equation}

Also notice that $(1\ra a)^{0} \leq 1\ra a^{0}$ and $(1\ra a)^{1}
\leq 1\ra a^{1}$. Assume that $(1\ra a)^{n} \leq 1\ra a^{n}$ for
some $n$. Taking into account (\ref{feq}) we prove the inequality
$(1\ra a)^{n+1} \leq 1\ra a^{n+1}$ as follows:
\[
\begin{array}
[c]{lllll}
(1\ra a)^{n+1}& = & (1\ra a)^{n}\cd (1\ra a) &  & \\
 & \leq & (1\ra a^{n})\cd (1\ra a)  &  &\\
 & \leq & 1\ra a^{n+1}. &  &
\end{array}
\]
\end{proof}

Let $A\in \V$ and $a,b \in A$. Let $(c,d),(u, w) \in R(a,b)$. It
follows from Remark \ref{r1} that there exists a pair of natural
numbers $(n,k)$ associated with $(c,d)$ and $(u,w)$. In
particular, $(n,2k)$ is also a pair of natural numbers associated
with $(c,d)$ and $(u,w)$. The pair $(n,2k)$ will be considered in
the following two lemmas.

\begin{lem} \label{l4}
Let $A\in \V$ and $a,b \in A$. Let $(c,d),(u, w) \in R(a,b)$.
There are natural numbers $n$ and $k$ that satisfy the following
conditions:
\begin{enumerate}[\normalfont \hspace{2mm} 1)]
\item$t_{n}^{2k}(a,b)\cd (c\ra u) \leq d\ra w$,
\item
$t_{n}^{2k}(a,b)\cd ((c\ra u) \we a \we b) \leq (d\ra w)\we a \we
b$,
\item $t_{n}^{2k}(a,b)\cd ((d\ra w) \we a \we b) \leq (c\ra
u)\we a \we b$,
\item $t_{n}^{2k}(a,b)\cd ((c\ra u) \vee a \vee b)
\leq (d\ra w)\vee a \vee b$,
 \item  $t_{n}^{2k}(a,b)\cd ((d\ra w)
\vee a \vee b) \leq (c\ra u)\vee a \vee b$,
 \item
$t_{n+1}^{2k}(a,b) \leq (c \ra u) \rla (d\ra w)$.
\end{enumerate}
\end{lem}

\begin{proof}
First note that $t_{n}^{k}(a,b)\leq c\ra d, d\ra c, u\ra w, w\ra
u$. Then
\[
\begin{array}
[c]{lllll}
 t_{n}^{2k}(a,b)\cd (c\ra u) & = & t_{n}^{k}(a,b)\cd t_{n}^{k}(a,b)\cd (c\ra u) &  & \\
 & \leq & (d\ra c)\cd (u\ra w) \cd (c\ra u)  &  &\\
 & = & (d\ra c)\cd (c\ra u) \cd (u\ra w)  &  &\\
 & \leq & d \ra w. &  &
\end{array}
\]
Thus,
\begin{equation} \label{eq1}
t_{n}^{2k}(a,b)\cd (c\ra u) \leq d\ra w,
\end{equation}
which is  condition 1).

We will prove that condition 1) implies condition 2). First note
that
\[
\begin{array}
[c]{lllll}
 t_{n}^{2k}(a,b)\cd ((c\ra u)\we a \we b) & \leq & t_{n}^{2k}(a,b)\cd (c\ra u)&  & \\
 & \leq & d \ra w. &  &
\end{array}
\]
Besides,
\[
t_{n}^{2k}(a,b)\cd ((c\ra u)\we a \we b)  \leq a\we b
\]
Then $t_{n}^{2k}(a,b)\cd ((c\ra u)\we a \we b)\leq (d \ra w)\we a
\we b$, which is condition 2). Similarly it can be proved that
\[
t_{n}^{2k}(a,b)\cd ((d\ra w) \we a \we b) \leq (c\ra u)\we a \we
b,
\]
which is condition 3).

Now we will prove the condition 4). It follows from (\ref{eq1})
that
\[
\begin{array}
[c]{lllll}
 t_{n}^{2k}(a,b)\cd ((c\ra u)\vee a \vee b) & = & (t_{n}^{2k}(a,b)\cd (c\ra u)) \vee (t_{n}^{2k}(a,b))\cd (a\vee b)) &  & \\
 & \leq & (d \ra w) \vee a \vee b, &  &
\end{array}
\]
which is  condition 4). Analogously we can show that
\[
t_{n}^{2k}(a,b)\cd ((d\ra w) \vee a \vee b) \leq (c\ra u)\vee a
\vee b,
\]
i.e.,  condition 5).

Finally we will prove  condition 6). It follows from display 
(\ref{eq1}) that
\begin{equation} \label{eq2}
(c\ra u) \ra ((c\ra u)\cd t_{n}^{2k}(a,b)) \leq (c\ra u) \ra (d\ra
w).
\end{equation}
Taking into account the inequality $1 \ra c_1 \leq c_2 \ra (c_1
\cd c_2)$ with $c_1 = t_{n}^{2k}(a,b)$ and $c_2 = c\ra u$ we
obtain that
\begin{equation} \label{eq3}
1 \ra t_{n}^{2k}(a,b) \leq (c\ra u) \ra ((c\ra u)\cd
t_{n}^{2k}(a,b)).
\end{equation}
Besides, it follows from Lemma \ref{il} that
\begin{equation} \label{eq4}
(1\ra t_{n}(a,b))^{2k} \leq 1 \ra t_{n}^{2k}(a,b).
\end{equation}
Since $t_{n+1}(a,b) \leq \square(a\rla b) \we \cdots \we
\square^{n+1}(a\rla b)$ and
\[
\square(a\rla b) \we \cdots \we \square^{n+1}(a\rla b) = 1 \ra
t_{n}(a,b),
\]
then $t_{n+1}(a,b)  \leq 1 \ra t_{n}(a,b)$, so $t_{n+1}^{2k}(a,b)
\leq (1\ra t_{n}(a,b))^{2k}$. Thus, it follows from (\ref{eq2}),
(\ref{eq3}) and (\ref{eq4}) that $t_{n+1}^{2k}(a,b) \leq (c \ra u)
\ra (d\ra w)$. Similarly we can show the inequality
$t_{n+1}^{2k}(a,b) \leq (d \ra w) \ra (c\ra u)$. Therefore,
\[
t_{n+1}^{2k}(a,b) \leq (c \ra u) \rla (d\ra w).
\]
\end{proof}

\begin{lem} \label{l5}
Let $A\in \V$ and $a,b \in A$. Let $(c,d)\in R(a,b)$ and $(u,w)\in
R(a,b)$. There exist natural numbers $n$ and $k$ that satisfy the
following conditions:
\begin{enumerate}[\normalfont \hspace{2mm} 1)]
\item $t_{n}^{2k}(a,b)((c\cd u)\we a \we b) \leq (d\cd w)\we a \we
b$, \item  $t_{n}^{2k}(a,b)((d\cd w)\we a \we b) \leq (c\cd u)\we a
\we b$, \item  $t_{n}^{2k}(a,b)((c\cd u)\vee a \vee b) \leq (d\cd
w)\vee a \vee b$, \item  $t_{n}^{2k}(a,b)((d\cd w)\vee a \vee b)
\leq (c\cd u)\vee a \vee b$, \item  $t_{n}^{2k}(a,b) \leq (c\cd
u)\rla (d\cd w)$.
\end{enumerate}
\end{lem}

\begin{proof}
First note that
\[
\begin{array}
[c]{lllll}
 t_{n}^{k}(a,b) \cd ((c\cd u)\we a \we b) & \leq & t_{n}^{k}(a,b) \cd (c\we a \we b) &  & \\
  & \leq & d\we a \we b  &  &\\
 & \leq & d &  &
\end{array}
\]
and
\[
\begin{array}
[c]{lllll}
 t_{n}^{k}(a,b) \cd ((c\cd u)\we a \we b) & \leq & t_{n}^{k}(a,b) \cd (u\we a \we b) &  & \\
  & \leq & w\we a \we b  &  &\\
 & \leq & w. &  &
\end{array}
\]
Then
\[
\begin{array}
[c]{lllll}
 t_{n}^{2k}(a,b)\cd ((c\cd u)\we a \we b) & \leq & t_{n}^{k}(a,b)\cd(c\we a \we b)\cd t_{n}^{k}(a,b)\cd(u\we a \we b) &  & \\
 & \leq & d \cd w. &  &
\end{array}
\]
This implies that $t_{n}^{2k}(a,b)\cd ((c\cd u)\we a \we b) \leq
(d\cd w)\we a \we b$, which is condition 1). Condition 2) can be
showed in an analogous way.

In order to prove 3), note that
\[
\begin{array}
[c]{lllll}
 t_{n}^{k}(a,b) \cd ((c\cd u)\vee a \vee b) & \leq & t_{n}^{k}(a,b) \cd (c\vee a \vee b) &  & \\
  & \leq & d\vee a \vee b  &  &
\end{array}
\]
and
\[
\begin{array}
[c]{lllll}
 t_{n}^{k}(a,b) \cd ((c\cd u)\vee a \vee b) & \leq & t_{n}^{k}(a,b) \cd (u\vee a \vee b) &  & \\
  & \leq & w\vee a \vee b.  &  &
\end{array}
\]
Hence,
\[
\begin{array}
[c]{lllll}
 t_{n}^{2k}(a,b)\cd ((c\cd u)\vee a \vee b) & \leq & t_{n}^{k}(a,b)\cd(c\vee a \vee b)\cd t_{n}^{k}(a,b)\cd(u\vee a \vee b) &  & \\
 & \leq & (d \vee a \vee b)\cd (w \vee a \vee b). &  &
\end{array}
\]
Straightforward computations show that
\[
(d \vee a \vee b)\cd (w \vee a \vee b) \leq (d\cd w) \vee a \vee
b.
\]
Thus,
\[
t_{n}^{2k}(a,b)((c\cd u)\vee a \vee b) \leq (d\cd w) \vee a \vee
b.
\]
So we have obtained  condition 3). Condition 4) is
similarly proved.

Finally we will prove  condition 5). Since $t_{n}^{k}(a,b)\leq
c\ra d$ and $t_{n}^{k}(a,b)\leq u\ra w$ then $t_{n}^{2k}(a,b)\leq
(c\ra d)\cd (u\ra w)$. Besides, it follows from Lemma
\ref{lemaprod} that $(c\ra d)\cd (u\ra w) \leq (c\cd u)\ra (d\cd
w)$, so $t_{n}^{2k}(a,b)\leq (c\cd u)\ra (d\cd w)$. Analogously it
can be proved that $t_{n}^{2k}(a,b)\leq (d\cd w)\ra (c\cd u)$.
Therefore, $t_{n}^{2k}(a,b) \leq (c\cd u)\rla (d\cd w)$.
\end{proof}

If $A$ is an algebra, $\theta$ a congruence and $a\in A$, then
$a/\theta$ denotes the equivalence class  of  $a$.

\begin{thm} \label{PT}
Let $A\in \V$ and $a,b \in A$. Then $\theta(a,b) = R(a,b)$.
\end{thm}

\begin{proof}
Notice that items 1)-5) from Lemma \ref{l4} are also true by
replacing $n$ by $n+1$. Then it follows from lemmas \ref{l3},
\ref{l4} and \ref{l5} that $R(a,b)$ is a congruence that contains
the pair $(a,b)$.

Let $\tau$ be a congruence such that $(a,b) \in \tau$. We will
prove that $R(a,b)\subseteq \tau$. Let $(c,d) \in R(a,b)$ and
$(n,k)$ a pair of natural numbers associated with $(c,d)$. Since
$(a,b)\in \tau$ then $(a\rla b,1) \in \tau$, which implies that
$(t_{n}^{k}(a,b),1) \in \tau$. Hence, it follows from
$\mathrm{(C1)}$ and $\mathrm{(C2)}$ that
\[
(c\we a \we b)/\tau \leq (d\we a \we b)/\tau,
\]
\[
(d\we a \we b)/\tau \leq (c\we a \we b)/\tau,
\]
\[
(c\vee a \vee b)/\tau \leq (d\vee a \vee b)/\tau,
\]
\[
(d\vee a \vee b)/\tau \leq (c\vee a \vee b)/\tau.
\]

Thus, $(c\we a\we b,d\we a\we b) \in \tau$ and $(c\vee a \vee b,
d\vee a\vee b) \in \tau$. Then, by Lemma \ref{l2} we conclude that
$(c,d)\in \tau$. Hence, $R(a,b)\subseteq \tau$. Therefore,
$\theta(a,b) = R(a,b)$.
\end{proof}

 The theorem  shows that $\V$ has locally equationally
definable principal congruences by the family of quaternary terms
QT introduced immediately after Remark \ref{rem:remark-1}.

The following result, which is \cite[Theorem 2.2]{SMc}, follows
from Theorem \ref{PT}. It shows that $\WH$ has locally
equationally definable principal congruences.

\begin{cor}
Let $A\in \WH$ and $a,b\in A$. Then $(c,d) \in \theta(a,b)$ if and
only if there exists a natural number $n$ that satisfies the
following conditions:
\begin{enumerate}[\normalfont \hspace{2mm} a)]
\item $c \we a \we b \we t_n(a,b) = d \we a \we b \we t_n(a,b)$,
\item$(c \vee a \vee b) \we t_n(a,b) = (d \vee a \vee b) \we
t_n(a,b)$, \item  $t_{n}(a,b) \leq c\rla d$.
\end{enumerate}
\end{cor}

The next corollary characterizes  the principal congruences of
the algebras of $\IDCRL$ using Theorem \ref{PT}.

\begin{cor} \label{corIDCRL}
Let $A\in \IDCRL$ and $a,b \in A$. Then $(c,d) \in \theta(a,b)$ if
and only if there exists a natural number $k$ such that $(a\rla
b)^{k} \leq c\rla d$.
\end{cor}

\begin{proof}
Let $A\in \IDCRL$, $a,b \in A$ and $n,k$ natural numbers. Since
$1\ra (a\rla b) = a \rla b$ then $t_{n}(a,b) = a\rla b$, so
$t_{n}(a,b)^{k} = (a\rla b)^{k}$. In what follows we will show
that condition (C3) implies conditions (C1) and (C2).

Assume that there exists a natural number $k$ such that
\[
(a\rla b)^{k} \leq c\rla d.
\]
Since $(a\rla b)^{k} \leq c\ra d$, then $c \cd (a\rla b)^{k} \leq
d$. Taking into account that $c\we a \we b \leq c$ we obtain that
\[
\begin{array}
[c]{lllll}
 (c\we a \we b)\cd (a\rla b)^{k} & \leq & c\cd (a\rla b)^{k} &  & \\
 & \leq & d. &  &
\end{array}
\]
Then,
\[
(c\we a \we b)\cd (a\rla b)^{k} \leq d,
\]
which implies that $(c\we a \we b)\cd (a\rla b)^{k} \leq d \we a
\we b$. Analogously it can be showed that $(d\we a \we b)\cd
(a\rla b)^{k} \leq c \we a \we b$. In order to prove that $(c\vee
a \vee b)\cd (a\rla b)^{k} \leq d\vee a \vee b$ first note the
equality
\[
(c\vee a \vee b)\cd (a\rla b)^{k} = (c\cd (a\rla b)^{k}) \vee
((a\vee b)\cd (a\rla b)^{k}).
\]
Since $c\cd (a\rla b)^{k} \leq d$ and $(a\vee b)\cd (a\rla
b)^{k}\leq a \vee b$ then
\[
(c\vee a \vee b)\cd (a\rla b)^{k} \leq d \vee a \vee b.
\]
Similarly, it can be showed that $(d\vee a \vee b)\cd (a\rla
b)^{k} \leq c \vee a \vee b$.

Hence, we have proved that condition (C3) implies conditions (C1)
and (C2). The rest of the proof follows from Theorem \ref{PT}.
\end{proof}

Corollary \ref{corIDCRL} can also be deduced from results due
to Agliano; more precisely in \cite{Ag} Agliano described the
principal congruences of BCI-monoids. It is part of the folklore
that if $(A,\we,\vee,\cd,\ra,1)$ is an integral commutative
residuated lattice then the congruences of
$(A,\we,\vee,\cd,\ra,1)$ coincide with the congruences of
$(A,\we,\cd,\ra,1)$, which is the underlying BCI-monoid of
$(A,\we,\vee,\cd,\ra,1)$. Hence, it follows from \cite[pp.\
409]{Ag} that if $\theta$ is a congruence of
$(A,\we,\vee,\cd,\ra,1)$ and $a,b \in A$, then $(c,d)\in
\theta(a,b)$ if and only if there is a natural number $k$ such
that $(a\rla b)^{k}\leq c\rla d$. Corollary \ref{corIDCRL} is a
particular case of the above mentioned property, when the
underlying lattice of $(A,\we,\vee,\cd,\ra,1)$ is distributive.

\section{Compatible functions}\label{SCF}

Let $A\in \V$. In this section we give a necessary and sufficient
condition for a function $f:A^{n} \ra A$ to be compatible. We also
find conditions on a binary function $g:A\times A \ra A$ that
imply that the function $a\mapsto$ $\mathrm{min}\{b\in A: g(a,b)\leq b\}$
is compatible when defined. We will employ similar ideas to those
used in \cite{CMS,ESM,SM,SMb,SMc}.

\begin{defn}
Let $A$ be an algebra and let $f: A^{n} \ra A$ a function.
\begin{enumerate}[\normalfont \hspace{2mm} 1.]
\item We say that $f$ is \textit{compatible} with a congruence $\theta$ of
$A$ if $(a_{i}, b_{i}) \in \theta$ for $i= 1,\ldots,n$ implies
$(f(a_{1},\ldots,a_{n}), f(b_{1},\ldots,b_{n})) \in \theta$. \item
We say that $f$ is a \textit{compatible function} of $A$ provided it is
compatible with all the congruences of $A$.
\end{enumerate}
\end{defn}

Let $A$ be an algebra and $f:A^{n} \ra A$ a function. Then $f$ is
compatible if and only if the algebras $A$ and $\left\langle
A,f\right\rangle$ have the same congruences. For $n = 1$, $f$ is
compatible if and only if $(f(a),f(b)) \in \theta(a,b)$ for every
$a,b \in A$. The simplest examples of compatible functions on an
algebra are the polynomial functions; note that in particular, all
term functions (and constant functions) are compatible \cite{Pix}.

\begin{prop} \label{pcom}
Let $A\in \V$ and $f:A\ra A$ a function. The following conditions
are equivalent:
\begin{enumerate}[\normalfont \hspace{2mm}1)]
\item $f$ is compatible. \item For every $a, b\in A$ there are
natural numbers $n$ and $k$ that satisfy the following
conditions:
\begin{enumerate}[\normalfont \hspace{2mm}(Cf1)]
\item $t_{n}^{k}(a,b)\cd (f(a) \we a \we b) \leq f(b) \we a \we
b$,
\item $t_{n}^{k}(a,b)\cd (f(a) \vee a \vee b) \leq f(b) \vee a
\vee b$,
\item $t_{n}^{k}(a,b) \leq f(a)\rla f(b)$.
\end{enumerate}
\end{enumerate}
\end{prop}

\begin{proof}
It follows from Theorem \ref{PT} and Remark \ref{r1}.
\end{proof}

\begin{rem}
Note that condition $\mathrm{(Cf3)}$ of Proposition \ref{pcom} can
be replaced by $t_{n}^{k}(a,b) \leq f(a)\ra f(b)$.
\end{rem}

Let $A$ be an algebra, $f: A^{n} \rightarrow A$ a function and
$\vec{a} = (a_{1},\ldots, a_{n}) \in A^{n}$. For $i = 1,\ldots,n$,
define the unary functions $f^{\vec{a}}_{i}: A \rightarrow A$ by
$$f^{\vec{a}}_{i}(x)= f(a_{1},...,a_{i-1}, x, a_{i+1},...,
a_{n}).$$ Then, we have the following characterization for the
compatibility of an $n$-ary function: $f$ is compatible if and
only if for every $\vec{a} \in A^{n}$ and every $i= 1,\ldots, n$,
the functions $f^{\vec{a}}_{i} : A \to A$ are compatible. Hence,
Proposition \ref{pcom} allows us to characterize compatible
$n$-ary functions in the variety $\V$.

\begin{cor} \label{coroc}
Let $A\in \WH$ and $f:A\ra A$ a function. Then $f$ is  compatible
if and only if for every $a,b \in A$ there exists $n\in
\mathbb{N}$  that satisfies the following conditions:
\begin{enumerate}[\normalfont \hspace{2mm} a)]
\item $f(a) \we a \we b \we t_n(a,b) = f(b) \we a \we b \we
t_n(a,b)$,
\item  $(f(a) \vee a \vee b) \we t_n(a,b) = (f(b) \vee a
\vee b) \we t_n(a,b)$,
\item  $t_{n}(a,b) \leq f(a)\rla f(b)$.
\end{enumerate}
\end{cor}

\begin{proof}
It follows from Proposition \ref{pcom}.
\end{proof}

Corollary \ref{coroc} was also proved in  \cite[Corollary
3.2]{SMc}.

\begin{cor} \label{corIDCRL2}
Let $A\in \IDCRL$ and $f:A\ra A$ a function. Then $f$ is
compatible if and only if for every $a,b\in A$ there is a natural
number $k$ such that $(a\rla b)^{k}\leq f(a)\rla f(b)$.
\end{cor}

\begin{proof}
It follows from Corollary \ref{corIDCRL}.
\end{proof}

The characterization of unary compatible functions for algebras in
$\IDCRL$ given in Corollary \ref{corIDCRL2} is exactly the
characterization of unary compatible functions given by Agliano in
\cite[pp.\ 410]{Ag} for BCI-monoids. Thus, the description of unary
compatible functions in $\IDCRL$ is also a direct consequence from
\cite{Ag}.

Independently from \cite{Ag}, Castiglioni, Menni and Sagastume
presented in \cite[Theorem 8]{CMS} a description of the compatible
functions in commutative residuated lattices. The unary case of
\cite[Theorem 8]{CMS} for the case of integral commutative
residuated lattices whose underlying lattice is distributive says
that if $A\in \IDCRL$ and $f:A\ra A$ is a function, then $f$ is
compatible if and only if for every $a,b\in A$ there is a natural
number $k$ such that $s(a,b)^{k}\leq s(f(a),f(b))$, where $s(a,b)
= (a\ra b)\cd (b\ra a)$. The proof of the above mentioned property
can be easily adapted in order to obtain Corollary
\ref{corIDCRL2}. \vspace{1pt}

Let $A\in \V$; if $g$ is  a binary function $g$ on $A$, we want to
find conditions implying that the function $a\mapsto$ min $\{b\in
A: g(a,b)\leq b\}$ is compatible whenever it is defined.

\begin{defn}
Let $A$ be a poset and let $g \colon A\times A \to A$ be a
function. We say that $g$ \textit{satisfies  condition} $\mathbf{(M)}$
if the following condition holds:
\[
\text{For all}\; a, b, c \in A, c\geq b\; \text{implies}\; g(a,c)
\leq g(a,b).
\]
\end{defn}

If $A$ is a $\vee$-semilattice and $g$ is a function that
satisfies  condition $\mathbf{(M)}$, then $g(a, g(a,b) \vee b)
\leq g(a,b) \vee b$ for every $a,b\in A$.

\begin{lem} \label{lec}
Let $A$ be a $\vee$-semilattice, and let $g \colon A\times A \to
A$ be a function that satisfies condition $\mathbf{(M)}$. The
following conditions are equivalent:
\begin{enumerate}[\normalfont \hspace{2mm} (a)]
\item There is a map $f \colon A \to A$ given by
$f(a) = \mathrm{min}\{b\in A:g(a,b) \leq b\}$. \item
There exists a map $h:A \ra A$ that satisfies the following
conditions for every $a,b\in A$:
\begin{enumerate}
\item[$\mathrm{(i)}$] $g(a,h(a)) \leq h(a)$,
\item[$\mathrm{(ii)}$] $h(a) \leq g(a,b) \vee b$.
\end{enumerate}
Moreover, in this case we have that $f = h$.
\end{enumerate}
\end{lem}

\begin{proof}
It follows from \cite[Lemma 15]{CMS}.
\end{proof}

Let $A$ be an algebra and $g:A\times A \ra A$ a function. For
every $a\in A$ define the function $g_a:A\ra A$ by $g_{a}(a,b) =
g(a,b)$. We say that \emph{$g$ is compatible in the first
variable} if $g_a$ is compatible for every $a\in A$. \vspace{1pt}

We apply Lemma \ref{lec} in order to prove the following
proposition.

\begin{prop} \label{min}
Let $A\in \V$ and let $f \colon A \to A$ be a function. The
following conditions are equivalent:
\begin{enumerate}[\normalfont \hspace{2mm} 1.]
\item $f$ is compatible. \item There
exists a function $g \colon A\times A \to A$ that satisfies
$\mathbf{(M)}$, compatible in the first variable and such that
$f(a) =$ min $\{b\in A:g(a,b) \leq b\}$. \item
There exists a function $\hat{g} \colon A\times A \to A$ that
satisfies $\mathbf{(M)}$, compatible in the first variable and
such that satisfies the following conditions for every $a,b\in A$:
\begin{enumerate}
\item[$\mathrm{(i)}$] $\hat{g}(a,f(a)) \leq f(a)$,
\item[$\mathrm{(ii)}$] $f(a) \leq \hat{g}(a,b) \vee b$.
\end{enumerate}
\end{enumerate}
Moreover, in this case we have that $g = \hat{g}$.
\end{prop}

\begin{proof}
Assume  condition 1., i.e., that $f$ is compatible. We define
$g \colon A\times A \to A$ by $g(a,b) = f(a)$. Hence,
condition $\mathrm{2.}$ is obtained. The equivalence between
$\mathrm{2.}$ and $\mathrm{3.}$ follows from Lemma \ref{lec}.

In order to show that  condition $\mathrm{3.}$ implies
condition $\mathrm{1.}$, let $a, b\in A$. Since $g$ is compatible
in the first variable then it follows from Proposition \ref{pcom}
that there are natural numbers $n$ and $k$ such that
\begin{equation}\label{eqone}
t_{n}^{k}(a,b)\cd (g(a,f(b)) \we a \we b) \leq g(b,f(b)) \we a \we
b,
\end{equation}
\begin{equation}\label{eqtwo}
t_{n}^{k}(a,b)\cd (g(a,f(b)) \vee a \vee b) \leq g(b,f(b)) \vee a
\vee b,
\end{equation}
\begin{equation}\label{eqthree}
t_{n}^{k}(a,b) \leq g(a,f(b))\ra g(b,f(b)).
\end{equation}
Taking into account (\ref{eqone}) and Lemma \ref{lec} we have that
\[
\begin{array}
[c]{lllll}
 t_{n}^{k}(a,b)\cd (f(a)\we a\we b) & \leq & t_{n}^{k}(a,b)\cd((g(a,f(b)) \vee f(b))\we a \we b) &  & \\
 & = & t_{n}^{k}(a,b) \cd ((g(a,f(b))\we a \we b) \vee (f(b)\we a \we b))  &  &\\
 & = & (t_{n}^{k}(a,b)\cd (g(a,f(b))\we a \we b)) \vee (t_{n}^{k}(a,b)\cd (f(b)\we a \we b))  &  &\\
 & \leq & (g(b,f(b))\we a \we b) \vee (f(b)\we a \we b)&  &\\
 & \leq & (f(b)\we a \we b) \vee (f(b)\we a \we b)&  &\\
 & = & f(b) \we a \we b. &  &
\end{array}
\]
Hence,
\begin{equation} \label{eqfour}
t_{n}^{k}(a,b)\cd (f(a)\we a\we b) \leq f(b) \we a \we b.
\end{equation}
Besides, it follows from (\ref{eqtwo}) and Lemma \ref{lec} that
\[
\begin{array}
[c]{lllll}
 t_{n}^{k}(a,b)\cd (f(a)\vee a\vee b) & \leq & t_{n}^{k}(a,b)\cd(g(a,f(b)) \vee f(b) \vee a \vee b) &  & \\
 & = & (t_{n}^{k}(a,b)\cd (g(a,f(b))\vee a \vee b)) \vee (t_{n}^{k}(a,b)\cd (f(b)\vee a \vee b))  &  &\\
 & \leq & (g(b,f(b))\vee a \vee b) \vee (f(b)\vee a \vee b)&  &\\
 & \leq & (f(b)\vee a \vee b) \vee (f(b)\vee a \vee b)&  &\\
 & = & f(b) \vee a \vee b. &  &
\end{array}
\]
Thus,
\begin{equation} \label{eqfive}
t_{n}^{k}(a,b)\cd (f(a)\vee a\vee b) \leq f(b) \vee a \vee b.
\end{equation}
Finally we will prove that $t_{n}^{k}(a,b)\leq f(a) \ra f(b)$. It
follows from (\ref{eqthree}) and the inequality $g(b,f(b))\leq
f(b)$ that
\begin{equation}\label{eqsix}
t_{n}^{k}(a,b) \leq g(a,f(b))\ra b.
\end{equation}
Since $f(a)\leq g(a,f(b))\vee f(b)$ then $f(a)\ra f(b) \geq
(g(a,f(b))\vee f(b)) \ra f(b)$. But $(g(a,f(b))\vee f(b)) \ra f(b)
= g(a,f(b))\ra f(b)$, so it follows from (\ref{eqsix}) that
\begin{equation} \label{eqseven}
t_{n}^{k}(a,b)\leq f(a) \ra f(b).
\end{equation}
Therefore, it follows from (\ref{eqfour}), (\ref{eqfive}),
(\ref{eqseven}) and Proposition \ref{pcom} that $f$ is a
compatible function.
\end{proof}

In the rest of this section we apply Proposition \ref{min} in
order to study possible generalizations of the gamma function
\cite[Example 5.1]{Caici}, the successor function \cite[Example
5.2]{Caici} and the Gabbay's function \cite[Example 5.3]{Caici}
considered by Caicedo and Cignoli in \cite{Caici} as examples of
implicit compatible operations on Heyting algebras. These
functions were also generalized in different frameworks, as for
instance in residuated lattices \cite{CMS,CSM} and in weak Heyting
algebras \cite{Cel-SM,SMc}. \vspace{1pt}

We start with the following definition that can be found in
\cite{C}.

\begin{defn}
Let $\Var$ be a variety of algebras of type $F$ and let $\epsilon
(C)$ be a set of identities of type $F \cup C$, where $C$ is a
family of new function symbols. We say that $\epsilon{(C)}$
defines \emph{implicitly} $C$, if in each algebra $A\in V$ there
is at most one family $\{f_{A}\colon A^{n} \rightarrow A\}_{f\in
C}$ such that $(A,f_{A})_{f\in C}$ satisfies the universal
closure of the equations in $\epsilon{(C)}$. In this case we say
that each $f$ is implicitly defined in $\Var$.
\end{defn}

In what follows we will consider $A\in \V$ and $n$ a natural
number.

\begin{ex}
Suppose that the underlying lattice of $A$ is bounded, and write
$0$ for the smallest element. We define the unary compatible
function $\gamma_{n}$ by
\[
\gamma_{n}(a) =  \mathrm{min}\{b\in A: a \vee \neg b^{n}\leq
b\},
\]
where $\neg x$ is defined by $x\ra 0$. Equivalently, $\gamma_{n}$
can be implicitly defined by the inequalities
\begin{enumerate}
\item[(g1)] $a \vee \neg(\gamma_{n}(a))^{n} \leq \gamma_{n}(a)$,
\item[(g2)]
$\gamma_{n}(a) \leq a \vee \neg b^{n} \vee b$.
\end{enumerate}
\end{ex}

The function $\gamma_n$ preserves the order, i.e., if $a\leq b$,
then $\gamma_{n}(a)\leq \gamma_{n}(b)$. In order to show it, let
$a\leq b$. By $\mathrm{(g2})$ we have that $\gamma_{n}(a)\leq a
\vee \neg(\gamma_{n}(b))^{n} \vee \gamma_{n}(b)$. Since $a\leq b$
then $\gamma_{n}(a)\leq b \vee \neg(\gamma_{n}(b))^{n} \vee
\gamma_{n}(b)$. Besides, by $\mathrm{(g1)}$ we obtain $b\vee
\neg(\gamma_{n}(b))^{n} \leq \gamma_{n}(b)$. Hence,
$\gamma_{n}(a)\leq \gamma_{n}(b)$.

\begin{lem}\label{lemgamma}
The function $\gamma_{n}$ is characterized as the unary function
that satisfies the following conditions for every $a,b$:
\begin{enumerate}
\item[$\mathrm{(g3)}$]  $\neg (\gamma_{n}(0))^{n} \leq
\gamma_{n}(0)$,
\item[$\mathrm{(g4)}$]  $\gamma_{n}(0) \leq b \vee
\neg b^{n}$,
\item[$\mathrm{(g5)}$]  $\gamma_{n}(a) = a \vee
\gamma_{n}(0)$.
\end{enumerate}
In particular, $\gamma_{n}$ is a polynomial function on $A$.
\end{lem}

\begin{proof}
Assume that $\gamma_n$ is an unary function that satisfies
$\mathrm{(g1)}$ and $\mathrm{(g2)}$. If we put $a = 0$ in
$\mathrm{(g1)}$ and $\mathrm{(g2)}$ then the equations
$\mathrm{(g3)}$ and $\mathrm{(g4)}$ follow. In what follows we
will prove $\mathrm{(g5)}$. By $\mathrm{(g2)}$ with $b =
\gamma_{n}(0)$ and by $\mathrm{(g3)}$ we obtain $\gamma_{n}(a)
\leq a \vee \gamma_{n}(0)$. In order to prove the other
inequality, note that it follows from $\mathrm{(g1)}$ that $a\leq
\gamma_{n}(a)$. Since $\gamma_{n}$ preserves the order then
$\gamma_{n}(0)\leq \gamma_{n}(a)$. Thus, $a\vee \gamma_{n}(0)\leq
\gamma_{n}(a)$. Hence, $\gamma_{n}(a) = a \vee \gamma_{n}(0)$,
which is $\mathrm{(g5)}$. Conversely, assume that $\gamma_n$ is a
unary function that satisfies $\mathrm{(g3)}$, $\mathrm{(g4)}$
and $\mathrm{(g5)}$. Condition $\mathrm{(g2)}$ follows from
$\mathrm{(g4)}$ and $\mathrm{(g5)}$. Finally we will prove
$\mathrm{(g1)}$. Since $\mathrm{(g5)}$ holds, then
$\gamma_{n}(0)\leq \gamma_{n}(a)$, so $(\gamma_{n}(0))^{n} \leq
(\gamma_{n}(a))^{n}$. Taking into account $\mathrm{(g3)}$ we
obtain $\neg (\gamma_{n}(a))^{n} \leq \neg (\gamma_{n}(0))^{n}
\leq \gamma_{n}(0)$. Thus, $a\vee \neg (\gamma_{n}(a)^{n}) \leq a
\vee \gamma_{n}(0) = \gamma_{n}(a)$. Therefore we have showed
condition $\mathrm{(g1)}$, which was our aim.
\end{proof}

It follows from Lemma \ref{lemgamma} that $\gamma_n$ is a
polynomial function, which implies that $\gamma_n$ is a compatible
function. Then we have obtained an alternative proof for the
compatibility of $\gamma_n$.

\begin{rem}
Let us write $\gamma$ for the function on Heyting algebras given
in \cite[Example 5.1]{Caici}. It was proved in \cite{FCSM} that
$\gamma$ can be defined by $\gamma_{n}(a) = \mathrm{min}\{b\in A:
a \vee \neg b\leq b\}$, or equivalently, as the unary function
that  satisfies the conditions $\mathrm{(g3)}$, $\mathrm{(g4)}$
and $\mathrm{(g5)}$ from Lemma \ref{lemgamma}. Thus, on Heyting
algebras we have that the definitions of $\gamma$ and $\gamma_1$
are the same.
\end{rem}

\begin{ex}
We define the unary compatible function $S_n$ by
\[
S_{n}(a) = \mathrm{min}\{b\in A: b^{n}\ra a\leq b\}.
\]
Equivalently, $S_{n}$ can be implicitly defined by the
inequalities
\begin{enumerate}
\item[(S1)] $(S_{n}(a))^{n} \ra a \leq S_{n}(a)$, \item[(S2)] $S_{n}(a) \leq b
\vee (b^{n}\ra a)$.
\end{enumerate}
\end{ex}

On Heyting algebras the function $S_1$ is the successor function.
For details about the successor function on Heyting algebras see
\cite{Esa,Kuz,LE}.

\begin{ex}
Assume that the underlying lattice of $A$ is bounded. Define the
unary compatible function $G_n$ by
\[
G_{n}(a) = \mathrm{min}\{b\in A: (b^{n}\ra a)\we \neg \neg a
\leq b\}.
\]
In an equivalent way, $G_{n}$ can be implicitly defined by the
inequalities
\begin{enumerate}
\item[(G1)] $((G_{n}(a))^{n} \ra a)\we \neg \neg a \leq G_{n}(a)$,
\item[(G2)]
$G_{n}(a) \leq b \vee ((b^{n}\ra a)\we \neg \neg a)$.
\end{enumerate}
\end{ex}

On Heyting algebras the function $G_1$ is the Gabbay's function,
which will be denoted $G$. The description of $G$ as the minimum of
certain set was proved in \cite{FCSM}. See also \cite{Gabbay} for
historical remarks about $G$.

\section{Other connections with existing literature}\label{SFR}

In \cite{Ce} Celani introduced distributive lattices with fusion
and implication. An algebra $(A,\we,\vee, \ra,0,1)$ of type
$(2,2,2,0,0)$ is a \emph{distributive lattice with implication} if
$(A,\we,\vee,0,1)$ is a bounded distributive lattice and for every
$a,b,c\in A$ the following conditions are satisfied:
\begin{enumerate}
\item[(I1)] $(a\ra b)\we(a\ra c)=a\ra(b\we c)$, \item[(I2)] $(a\ra
c)\we(b\ra c)=(a\vee b)\ra c$, \item[(I3)] $0\ra a = a \ra 1 = 1$.
\end{enumerate}
An algebra $(A,\we,\vee, \cd,0,1)$ of type $(2,2,2,0,0)$ is a
\emph{distributive lattice with fusion} if $(A,\we,\vee,0,1)$ is a
bounded distributive lattice and for every $a,b,c\in A$ the
following conditions are satisfied:
\begin{enumerate}
\item[(F1)] $a\cd (b\vee c) = (a\cd b) \vee (a\cd c)$, \item[(F2)]
$(b\vee c)\cd a = (b \cd c) \vee (a\cd c)$, \item[(F3)] $0\cd a =
a\cd 0 = 0$.
\end{enumerate}
An algebra $(A,\we,\vee,\cd,\ra,0,1)$ is a \emph{bounded
distributive lattice with fusion and implication} if $(A,\we,\vee,
\ra,0,1)$ is a bounded distributive lattice with implication and
$(A,\we,\vee, \cd,0,1)$ is a bounded distributive lattice with
fusion. Straightforward computations show that for every
$(A,\we,\vee,\cd,\ra,1) \in \V$ with a smallest element $0$ we
have that the algebra $(A,\we,\vee,\cd,\ra,0,1)$ is a distributive
lattice with fusion and implication.

There are other connections of the present paper with existing
literature. A \emph{generalized commutative residuated lattice}
\cite[Def.\ 1.1]{SMb} is an algebra $(A, \wedge, \vee, \cd, \ra,
e)$ of type $(2,2,2,2,0)$ that  satisfies the following
conditions: $(A,\cd ,e)$ is a commutative monoid, $(A, \vee,
\wedge)$ is a lattice, and for every $a,b, c \in A$ the following
properties hold: $a\ra(b\we c) = (a\ra b)\we(a\ra c)$, $(a\vee
b)\ra c = (a\ra c)\wedge(b\ra c)$, $(a\ra b)\cd(b\ra c)\leq a\ra
c$, and $e\leq a\ra a$. We write $\GCRL$ for the variety of
generalized commutative residuated lattices. Clearly, $\V$ is a
subvariety of $\GCRL$.

\begin{rem}
Let $A\in \GCRL$, $a\in A$ and $n\geq 1$. We define $\square(a): =
e\ra a$. As usual, also define $\square^{2}(a) = \square(\square
(a))$ and $a^{2} = a \cd a$. Notice that if $A\in \V$, the
definition of $\square$ given before collapses to the definition
of $\square$ given in Section 2 for algebras of $\V$.
\end{rem}

\begin{defn} \cite[Def.\ 1.2]{SMb}
A \emph{commutative weak residuated lattice} is a generalized
commutative residuated lattice $(A, \wedge, \vee, \cd, \ra, e)$
that satisfies the following conditions for every $a,b,c \in A:$
\begin{enumerate}[\normalfont \hspace{3mm} (R1)]
\item $a\cd(a\ra b) \leq b$, \item $(a\cd b) \vee (a \cd c) = a
\cd (b\vee c)$, \item $\square(a) \leq b\ra (b \cd a)$, \item
$a\leq (a\ra e)\ra e$, \item $\square(a^{2}) \leq \square(a) \cd
\square(a)$, \item $\square(a) \ra (\square(a)\ra e) \leq
\square^{2}(a) \ra e$, \item $(\square(a) \ra e) \cd(\square(a)
\ra e) \leq \square(a^{2}) \ra e$.
\end{enumerate}
\end{defn}

We write $\CWRL$ for the variety of commutative weak residuated
lattices. Commutative weak residuated lattices are a common
abstraction of commutative residuated lattices \cite{Ts} and weak
Heyting algebras that satisfy $a\we (a\ra b)\leq b$ for every
$a,b$. The varieties $\V$ and $\CWRL$ are incomparable. In order
to show it, first note that there are commutative residuated
lattices without largest element (or with underlying lattice not
necessarily distributive), so $\CWRL \nsubseteq \V$. Besides,  the
algebras of $\CWRL$ satisfy the inequality $a\cd (a\ra b)\leq b$
for every $a,b$. Since the previous inequality is not satisfied by
the algebra given in Example \ref{ex1}, then $\V \nsubseteq
\CWRL$.

\subsection*{Acknowledgments}
This project has received funding from the European Union's
Horizon 2020 research and innovation programme under the Marie
Sklodowska-Curie grant agreement No. 689176.

The first author was also partially supported by the research
grant 2014 SGR 788 from the government of Catalonia and by the
research projects MTM2011-25747 and MTM2016-74892-P from the
government of Spain, which includes \textsc{feder} funds from the
European Union and he also acknowledges financial support from the
Spanish Ministry of Economy and Competitiveness, through the
``Mar\'ia de Maeztu'' Programme for Units of Excellence in R\&D
(MDM-2014-0445). The second author was also supported by CONICET
Project PIP 112-201501-00412.


\newpage

-----------------------------------------------------------------------------------------
\\
Ramon Jansana,\\
Barcelona Graduate School of Mathematics\\
Philosophy Department
Universitat de  Barcelona.\\
Montalegre, 6,\\
08001, Barcelona,\\
España.\\
jansana@ub.edu

---------------------------------------------------------------------------------------
\\
Hern\'an Javier San Mart\'in,\\
Departamento de Matem\'atica, \\
Facultad de Ciencias Exactas (UNLP), \\
and CONICET.\\
Casilla de correos 172,\\
La Plata (1900),\\
Argentina.\\
hsanmartin@mate.unlp.edu.ar

\end{document}